\documentclass[10pt,reqno]{amsart}

\usepackage{amsmath}
\usepackage{amsthm}
\usepackage{amssymb}
\usepackage{amsfonts}
\usepackage{latexsym}

\usepackage{mathtools}
\mathtoolsset{showonlyrefs}

\usepackage{todonotes}

\usepackage{tracefnt}
\usepackage{anyfontsize}
\usepackage{multicol}
\usepackage{graphicx}
\usepackage[size=footnotesize,labelfont=bf]{caption}
\usepackage{subcaption}
\usepackage{hyperref}
\usepackage{url, cite, enumitem}

\usepackage[sc]{mathpazo}

\usepackage{mathrsfs}
\usepackage{mathtools}
\usepackage{microtype}
\usepackage[T1]{fontenc}

\DeclarePairedDelimiter\floor{\lfloor}{\rfloor}
\DeclarePairedDelimiter\abs{\lvert}{\rvert}

\renewcommand{\vec}[1]{\mathbf{#1}}

\newcommand{\Z}{\mathbb{Z}}
\newcommand{\F}{\mathbb{F}}

\newcommand{\K}{\mathbb{K}}
\newcommand{\N}{\mathbb{N}}
\newcommand{\Q}{\mathbb{Q}}

\newcommand{\diag}{\operatorname{diag}}
\newcommand{\Tr}{\operatorname{Tr}}
\newcommand{\tr}{\operatorname{tr}}
\newcommand{\nm}{\mathbf{N}}

\theoremstyle{plain}
\newtheorem{theorem}[equation]{Theorem}
\newtheorem{lemma}[equation]{Lemma}

\newtheorem{corollary}[equation]{Corollary}
\theoremstyle{remark}
\newtheorem{remark}[equation]{Remark}
\theoremstyle{definition}

\newtheorem{example}[equation]{Example}

\allowdisplaybreaks{}
\mathtoolsset{showonlyrefs=true}

\renewcommand{\pmod}[1]{\,\,(\operatorname{mod} #1)}
\usepackage{enumitem}
\setlist[enumerate]{leftmargin=*}
\setlist[itemize]{leftmargin=*}
\setlist[enumerate,1]{label=(\alph*),font=\upshape}
\setlist[enumerate,2]{label=(\roman*),font=\upshape}
\let\oldenumerate=\enumerate
	\def\enumerate{
	\oldenumerate
	\setlength{\itemsep}{3pt}
	}
\let\olditemize=\itemize
	\def\itemize{
	\olditemize
	\setlength{\itemsep}{3pt}
	}

\begin{document}
\title[Moments of Gaussian Periods and Modified Fermat Curves]{Moments of Gaussian Periods\\and Modified Fermat Curves}

\author{Stephan Ramon Garcia}
\address{Department of Mathematics and Statistics, Pomona College, 610 N. College Ave., Claremont, CA 91711} 
\email{stephan.garcia@pomona.edu}
\urladdr{\url{https://stephangarcia.sites.pomona.edu/}}

\author{Brian Lorenz}
\address{Department of Astronomy, University of California at Berkeley, 501 Campbell Hall \#3411, Berkeley, CA 94720} 
\email{brian.lorenz@berkeley.edu}

\author{George Todd}
\address{Department of Mathematics, University of Dayton, 300 College Park, Dayton, OH 45469}
\email{gtodd1@udayton.edu}

\begin{abstract}
We use supercharacter theory to study moments of Gaussian periods.  For $p-1=dk$ and fixed $k$, we compute the fourth absolute moments for all but finitely many primes $p$. For $d$ fixed, we relate the fourth absolute moments to the number of rational points on modified Fermat curves. For small $d$, this relation is in terms of a single curve. For larger $d$, we provide both exact formulas using families of modified Fermat curves and bounds via Hasse--Weil.
\end{abstract}

\thanks{First author supported by a David L. Hirsch III and Susan H. Hirsch Research Initiation Grant, the Institute for Pure and Applied Mathematics (IPAM) Quantitative Linear Algebra program, and NSF Grants DMS-1800123 and DMS-2054002}


\maketitle

\section{Introduction} 

The study of Gaussian periods was initiated by Gauss in his \emph{Disquisitiones Arithmeticae}, where he used them to show the constructibility of the heptadecagon \cite{gauss}.  These character sums have many applications, including solutions to congruences over finite fields, the distribution of power residues, pseudo-random number generators \cite{igor}, Kummer's development of arithmetic in the cyclotomic integers~\cite{Kolmogorov}, and the optimized AKS primality test of H.W.~Lenstra and C.~Pomerance~\cite{Agrawal,Lenstra}.
Gaussian periods can also be defined, with some restrictions, for composite moduli.
In this context, striking visualizations have recently emerged \cite{GNGP, GGP,GHM, LutzSolo}, prompting
E.~Eischen and N.~Milnes
 to create an app just for this purpose \cite{GPSoft}.  Similar visual phenomena have 
also been observed in the context of Kloosterman-type sums \cite{GKS} and exponential sums associated to the symmetric group \cite{GNSG}.

Let $p$ be an odd prime, $g$ a fixed primitive root modulo $p$, and $e_p(x) = e^{2\pi i x/ p}$.
A \emph{Gaussian period modulo $p$} is the character sum
\begin{equation*}
    \eta_a = \sum_{j = 0}^{k-1} e_p(g^{dj+a}),
\end{equation*}
in which $p = dk + 1$; see \cite{berndt_evans, gauss_jacobi}.
These are closely related to Gauss sums: if $\chi$ is a character of order $d$ modulo $p$, then the \emph{Gauss sums}
\begin{equation*}
    h(d) = \sum_{n=0}^{p-1} e_p(n^d)
    \qquad \text{and} \qquad
    G(\chi) = \sum_{n=0}^{p-1} \chi(n) e_p(n)
\end{equation*}
satisfy
\begin{equation*}
1+d\eta_0 = h(d) = \sum_{j=1}^{d-1} {G(\chi^j)} \,.
\end{equation*}

If $\K$ is the unique intermediate field of degree $d$ in $\Q(\zeta_p) / \Q$, where $\zeta_p = e_p(1)$, then $\K = \Q(\eta_0)$, so the $d$ periods $\eta_0, \eta_1,\ldots, \eta_{d-1}$ are conjugate and
$\eta_0 = \tr_{\Q(\zeta_p)/\K}(\zeta_p)$.
Consequently, Gaussian periods are also called \emph{cyclotomic periods} or \emph{$k$-nomial periods}.

For $n\in \N$, the \emph{$n$th absolute moment} (of the Gaussian periods modulo $p$) is
\begin{equation}
  V_n(p) = \sum_{s = 0}^{d-1} \left| \eta_s \right|^n.
\end{equation}
The absolute value is present to ensure nonnegativity since Gaussian periods may be non-real.
While the evaluation of the (non-absolute) power moments $\sum_{s=0}^{d-1} \eta_s$ and $\sum_{s=0}^{d-1} \eta_s^2$ is elementary, the third is not: Hurwitz \cite{hurwitz} showed that for $d$ prime
\begin{equation*}
	N(d, p) = \frac{1}{p} \bigg( (p-1)^3 + (p-1)d^2\sum_{i=0}^{d-1} \eta_i^3 \bigg),
\end{equation*}
where $N(d,p)$ is the number of nontrivial solutions to $x^d + y^d + z^d \equiv 0 \pmod p$.
Thus, $N(d,p) > 0$ whenever
\begin{equation*}
	V_3(p) = \sum_{i=0}^{d-1} \abs{\eta_i}^3 < \Big( \frac{p-1}{d} \Big)^2 \,.
\end{equation*}
This can be generalized \cite{gauss_jacobi}:
if $N(n, d, p)$ is the number of nontrivial solutions to $x_1^d + x_2^d+\cdots + x_n^d \equiv 0 \pmod p$, then
\begin{equation*}
  N(n,d,p) = p^{n-1} + \frac{p-1}{pd} \sum_{s=0}^{d-1} (1 + d\eta_s)^n \,.
\end{equation*}

In this note, we use supercharacter theory to reduce the computation of the absolute moments of Gaussian periods to counting the number of $\F_p$-rational points on the modified Fermat curves $ax^d + by^d = cz^d$ over $\mathbb{P}^2(\F_p)$.
The use of supercharacters (see Subsection \ref{Subsection:Supercharacters}) to study Gaussian periods was initiated by Duke, Garcia, and Lutz \cite{GNGP}, and continued by Garcia, Hyde, and Lutz in \cite{GHM, LutzSolo}.

We study $V_n(p)$ from two perspectives: fixed $d$ and fixed $k$.
Such an approach has been undertaken before for related quantities.
In \cite{myerson1,myerson2}, Myerson studied
\begin{equation*}
  \prod_{j=0}^{d-1} \eta_j = \nm_{\K/\mathbb{Q}}\big( \Tr_{\mathbb{Q}(\zeta_p)/\K} \zeta_p \big)
\end{equation*}
for fixed $k$ and fixed $d$ with $p \equiv 1 \pmod 3$.
With respect to cyclotomic field extensions, fixed $d$ corresponds to fixing the degree of $\K/\mathbb{Q}$ and fixed $k$ corresponds to fixing the degree of $\mathbb{Q}(\zeta_p)/\K$.
In terms of supercharacter theory \cite{SESUP}, fixed $d$ corresponds to fixing the number of superclasses and fixed $k$ corresponds to fixing the size of superclasses.
Other results on moments of character sums are in \cite{mrss}.

We need a little more terminology before we state our first result.
A pair $(p,k)$ is \emph{circular} if $dk = p - 1$ and the subgroup $\Gamma = \langle g^d \rangle$ of $\F_p^\times$ of order $k$ satisfies $| ( \Gamma a + b) \cap (\Gamma c + d) | \leq 2$ for all $a, c \in \F_p^\times$ and $b, d \in \F_p$ with $\Gamma a \neq \Gamma c$ or $b \neq d$.
The name ``circular'' is taken from the field of planar nearrings, where
$E_c^r = \left\{ \Gamma r + b \mid b \in \Gamma c \right\}$
is viewed as a set of circles with radius $r$ centered at $\Gamma c$.
The circularity condition ensures that any two distinct ``circles'' in $E_c^r$ intersect in $0$, $1$, or $2$ points \cite{ke_kiechle_jctsa}.
For a fixed $k$, all but finitely many primes $p$ are circular \cite{clay, fong_ke}.

We now state the main results of the paper.
For fixed $k$, we find $V_4(p)= \sum_{s = 0}^{d-1} \left| \eta_s \right|^4$ explicitly for all but finitely many primes $p$.

\begin{theorem}\label{thm:fixk}
  Let $(p, k)$ be circular.
  \begin{enumerate}
  	\item If $2 \mid k$, then $ V_4(p) = 3p(k-1) - k^3$.

  	\item If $2 \nmid k$ and $(p, 2k)$ is circular, then $ V_4(p) = p(2k-1) - k^3$.
  \end{enumerate}
\end{theorem}

For this result, in addition to the method of supercharacters, we use techniques similar to those in \cite{ke_kiechle_pams}, where the authors compute the number of solutions to $x^m + y^m - z^m = 1$ over a finite field given a certain circularity condition.

We use supercharacter theory to compute $V_4(p)$ explicitly for $d=3$ and $d=4$:

\begin{theorem} \label{thm:v4d3}
  If $d = 3$, then
 \begin{equation}
    V_4(p) = \frac{1}{27}\big(10p^2 + 4\left(4 - M_3 \right)p + 1\big),
  \end{equation}
 where $M_3$ is the number of projective $\F_p$-rational points on the Fermat curve $x^3 + y^3 = z^3$.
  Moreover,
 \begin{equation}
    \big| 27V_4(p) - (6p^2 + 12p + 1) \big| \leq 8p^{3/2}.
  \end{equation}
\end{theorem}

The $\delta$-symbol in the next two theorems takes the value $1$ if the condition in the subscript is satisfied, and the value $0$ otherwise.

\begin{theorem}\label{thm:d4p1p8}
  If $d = 4$, then
  \begin{equation}\small
      V_4(p) = \begin{cases}
	  \dfrac{1}{576} \big( p^3 + (191-2M_{4,1})p^2 + (199 - 38M_{4,1}+M_{4,1}^2)p + 9 \big) &\text{if $p \equiv 1 \pmod 8$}, \\[10pt]
	  \dfrac{1}{576} \big( p^3 + (71-2M_{4,2})p^2 + (79+10M_{4,2}+M_{4,2}^2)p + 9 \big) &\text{if $p \equiv 5 \pmod 8$},
      \end{cases}
  \end{equation}
  where $M_{4,1}$ and $M_{4,2}$ are the number of projective $\F_p$-rational points on the curves $x^4 + y^4 = z^4$ and $x^4 + y^4 = g^2z^4$, respectively, in which $g$ is a generator of the unique subgroup of $\F_p^\times$ of order $k$. 
\end{theorem}

For all $d$, we obtain completely explicit bounds on $V_4(p)$; see Figure \ref{Figure:Only}.

\begin{theorem}\label{thm:all_d_intro}
    If $d > 2$, then
\begin{itemize}
  \item[(a)] $V_4(p)   \, \geq \,  \frac{1}{d^3} \big( (\delta_{2d\mid(p-1)} d-1)p+1 \big)^2$ \,,
  \item[(b)] $\small V_4(p)  \leq 
	\begin{cases} 
	    \dfrac{1}{d^3} \big( (d-1)(d^2 -3d + 3)p^2 + 4(d-1)(d-2)p^{3/2} + 6(d-1)p + 1 \big) & \text{if $2d \mid (p-1)$}, \\[8pt]
	    \dfrac{1}{d^3} \big( (d^3-5d^2+8d-3)p^2 + 4(d-2)p^{3/2} + ( \frac{d^3-2d^2-2d+6}{d-1} ) p + 1 \big) & \text{otherwise}.
	    \end{cases}$
\end{itemize}
\end{theorem}

\begin{figure}
 \centering
    \includegraphics[width=0.475\textwidth]{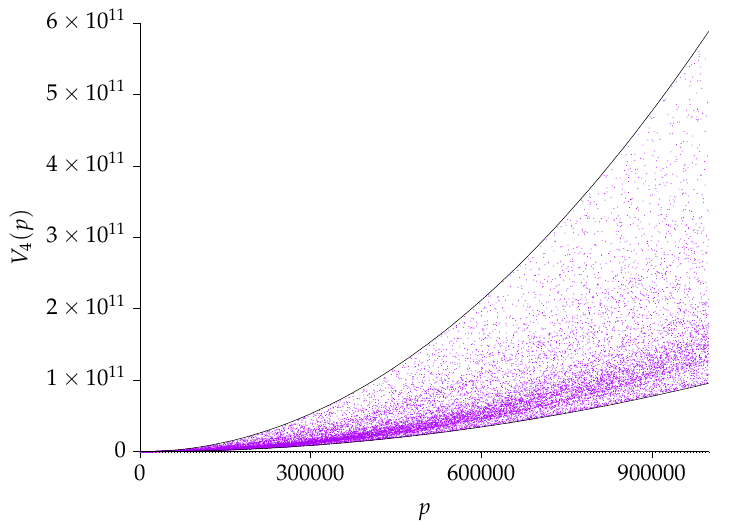}\hfill
        \includegraphics[width=0.475\textwidth]{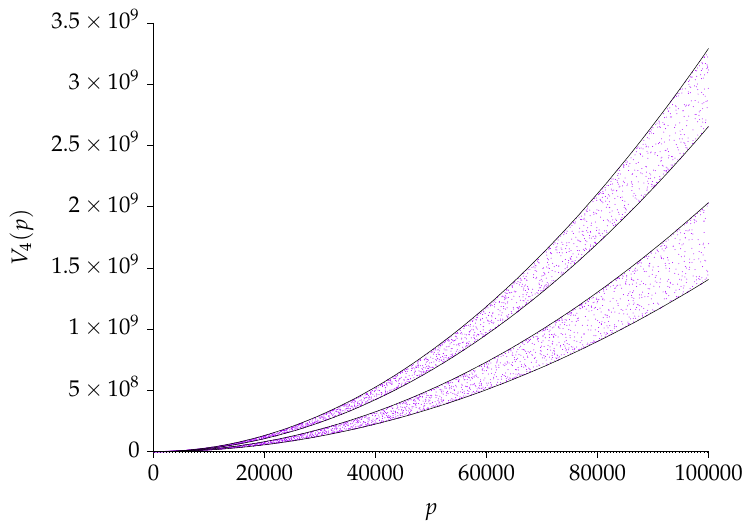}
 \caption{(\textsc{Left}) $V_4(p)$ for $d = 8$ with the bounds from Theorem~\ref{thm:all_d_intro}.
 (\textsc{Right}) $V_4(p)$ for $d = 4$. Lower bounds are from Theorem~\ref{thm:d4p1p8} via \eqref{eq:LowerSpecial} and upper bounds are from Theorem~\ref{thm:all_d_intro} (upper band is the case $p \equiv 1 \pmod 8$ and lower band is the case $p \equiv 5 \pmod 8$).}
\label{Figure:Only}
\end{figure}

\begin{remark}
For $d=4$, Theorem \ref{thm:d4p1p8} and the Hasse--Weil bound yield
\begin{equation}\label{eq:LowerSpecial}
  \frac{1}{64} \big( (3p+1)^2 + 8p(p-3)\delta_{8 \mid (p-1)} \big) \leq V_4(p) \,,
\end{equation}
which is stronger than the lower bound in Theorem \ref{thm:all_d_intro}.  
\end{remark}

The organization of this paper is as follows: Section \ref{Section:Background} contains preliminary results needed for the remainder of the paper. The proof of Theorem~\ref{thm:fixk} is in Section \ref{sec:fapmfk} and proofs of Theorems~\ref{thm:v4d3}, \ref{thm:d4p1p8}, and \ref{thm:all_d_intro} are in Section \ref{Section:FixedD}.

We thank the anonymous referee for many helpful suggestions and corrections.

\section{Background and preliminaries}\label{Section:Background}

\subsection{Background on supercharacters}\label{Subsection:Supercharacters}
Diaconis and Isaacs introduced supercharacters axiomatically in 2008~\cite{Diaconis}, building upon 
work of C.~Andr\'{e}~\cite{An95, An01}. Supercharacters arise in the study of Hopf algebras and symmetric functions~\cite{Aguiar}, random walks on upper triangular matrices~\cite{ACDiSt04}, and Schur rings~\cite{DiTh09, ThVe09, Th10}.
A variety of exponential sums are values of supercharacters on abelian groups \cite{SESUP}.
For example, Ramanujan sums can be given a unified presentation via supercharacters \cite{RSS}.
The first and third authors used supercharacters to establish a sixth-moment formula for Kloosterman sums and
a connection to elliptic curves \cite{K6M}.  This was extended to the eighth moment by Sayed and Kalita \cite{sayed2021eighth}; see also
\cite{Sayed2}.

Let $\Gamma$ be a subgroup of $\mathrm{GL}_d( \Z/n\Z)$ that is closed under the transpose 
and let $X_1,X_2,\ldots,X_N$ denote the orbits in $G=(\Z/n\Z)^d$ under the action of $\Gamma$.
Define
\begin{equation}\label{eq:Supercharacter}
  \sigma_i(\vec{y}) = \sum_{\vec{x} \in X_i} e_n( \vec{x} \cdot \vec{y}),
\end{equation}
where $\vec{x}\cdot \vec{y}$ is the formal dot product of $\vec{x}, \vec{y} \in (\Z/n\Z)^d$ and $e_n(x) = \exp(2\pi i x/n)$.
These are \emph{supercharacters} on $(\Z/n\Z)^d$ and the $X_i$ are \emph{superclasses}.
Supercharacters are constant on superclasses, so we write $\sigma_i(X_j)$ and construct the  $N \times N$ matrix
\begin{equation}\label{eq:USCT}
  U = \frac{1}{\sqrt{n^d}} \left[  \frac{   \sigma_i(X_j) \sqrt{  |X_j| }}{ \sqrt{|X_i|}} \right]_{i,j=1}^N ,
\end{equation}
which is self transpose and unitary \cite[Lem.~1]{SESUP}, \cite[Sect.~2.1]{RSS}.
It represents, with respect to a certain orthonormal basis, the restriction of the discrete Fourier transform 
to the subspace of $L^2(G)$ consisting of functions that are constant on each superclass \cite{DCT}.

The next lemma identifies the matrices diagonalized by the unitary matrix \eqref{eq:USCT}.
The proof is similar to a result from classical character theory \cite[Section 33]{CR62}.
A more general version of this lemma, in which $G$ need not be abelian, is \cite[Thm.~4.2]{RSS}.
The version below is \cite[Thm.~2]{SESUP}.

\begin{lemma}\label{Lemma:SESUP}
  Let $\Gamma = \Gamma^{\mathsf{T}}$ be a subgroup of $\mathrm{GL}_d(\Z/n\Z)$, let $\{X_1,X_2,\ldots,X_N\}$ denote the set of $\Gamma$-orbits in $G = (\Z/n\Z)^d$ induced by the action of $\Gamma$, and let $\sigma_1,\sigma_2,\ldots,\sigma_N$ denote the corresponding supercharacters \eqref{eq:Supercharacter}.
  For each fixed $\vec{z}$ in $X_k$, let $c_{i,j,k}$ denote the number of solutions $(\vec{x}_i,\vec{y}_j) \in X_i \times X_j$ to the equation $\vec{x}+\vec{y} = \vec{z}$.
  \begin{enumerate}
    \item $c_{i,j,k}$ is independent of the representative $\vec{z}$ in $X_k$ that is chosen.
    \item The identity $\displaystyle\sigma_i(X_{\ell}) \sigma_j(X_{\ell}) = \sum_{k=1}^N c_{i,j,k} \sigma_k(X_{\ell})$
      holds for $1\leq i,j,k,\ell \leq N$.

    \item The matrices $T_1,T_2,\ldots,T_N$, whose entries are given by
      \begin{equation*}
        [T_i]_{j,k} = \frac{ c_{i,j,k} \sqrt{ |X_k| } }{ \sqrt{ |X_j|} },
      \end{equation*}
      each satisfy $T_i U = U D_i$,
      in which $\displaystyle D_i = \operatorname{diag}(\sigma_i(X_1), \sigma_i(X_2),\ldots, \sigma_i(X_N) )$.
      In particular, the $T_i$ are simultaneously unitarily diagonalizable.

    \item Each $T_i$ is normal (that is, $T_i^*T_i = T_i T_i^*$) and $\{T_1,T_2,\ldots,T_N\}$ is a basis for the commutative algebra of all $N \times N$ complex matrices $T$ such that $U^*TU$ is diagonal.
    \end{enumerate}
\end{lemma}

\subsection{Gaussian periods as supercharacters}
We now describe a supercharacter theory that gives rise to Gaussian periods.
The action of $\Gamma = \langle g^d \rangle$ on $\F_p$ produces $d+1$ superclasses
$X_0      =      \Gamma$,
$X_1      =       g\Gamma$,
$X_2      =       g^2\Gamma$,\ldots,     
$X_{d-1}  =    g^{d-1}\Gamma$,
and $X_{d}    =       \{0\}$.
Let $0 \leq i, j \leq d-1$, so that $X_i = g^i \Gamma$.
Choose $y = g^j \in X_j$ and compute:
\begin{align*}
  \sigma_i(X_j)
    &= \sum_{x \in X_i} e_p( xy ) 
    = \sum_{x \in g^i\Gamma} e_p( xy ) 
    = \sum_{z \in y g^i\Gamma} e_p(z) 
    = \sum_{z \in g^{i+j}\Gamma} e_p(z) = \eta_{i+j}.
\end{align*}
This, together with \eqref{eq:USCT}, yields the $(d+1) \times (d+1)$ symmetric unitary matrix
\begin{equation}\label{eq:Ugp}
  U = \small
  \frac{1}{\sqrt{p}}
  \left[
    \begin{array}{cccc|c}
				\eta_0                &  \eta_1                &  \cdots  &  \eta_{d-1}            &  \sqrt{\frac{p-1}{d}}  \\[2pt]
				\eta_1                &  \eta_2                &  \cdots  &  \eta_0                &  \sqrt{\frac{p-1}{d}}  \\[2pt]
				\vdots                &  \vdots                &  \ddots  &  \vdots                &  \vdots                \\[2pt]
				\eta_{d-1}            &  \eta_0                &  \cdots  &  \eta_{d-2}            &  \sqrt{\frac{p-1}{d}}  \\[2pt]
				\hline
				\sqrt{\frac{p-1}{d}}  &  \sqrt{\frac{p-1}{d}}  &  \cdots  &  \sqrt{\frac{p-1}{d}}  &  1                     \\[2pt]
    \end{array}
  \right] \,.
\end{equation}
Setting $D = D_0$ gives
\begin{equation}\label{eq:Dgp}
  D = \diag\big(\eta_0, \eta_1, \ldots, \eta_{d-1}, \tfrac{p-1}{d} \big) \,.
\end{equation}
Finally, let $\alpha$ be the least positive residue of $(p-1)/2$ modulo $d$, and define $T = T_0$.
Then from Lemma~\ref{Lemma:SESUP} we deduce that
\begin{equation}\label{eq:Tgp}
  T=\small
  \left[
    \begin{array}{cccc|c}
				c_{0,0,0}             &      c_{0,0,1}       &       \cdots  &  c_{0,0,d-1}       &  0\\
				c_{0,1,0}             &      c_{0,1,1}       &       \cdots  &  c_{0,1,d-1}       &  0\\
				\vdots                &      \vdots          &       \ddots  &  \vdots            &  \vdots\\
				c_{0,\alpha,0}        &      c_{0,\alpha,1}  &       \cdots  &  c_{0,\alpha,d-1}  &  \sqrt{\frac{p-1}{d}}\\
				\vdots                &      \vdots          &       \ddots  &  \vdots            &  \vdots\\
				c_{0,d-1,0}           &      c_{0,d-1,1}     &       \cdots  &  c_{0,d-1,d-1}     &  0\\
				\hline
				\sqrt{\frac{p-1}{d}}  &      0               &       0       &  0                 &  0\\
    \end{array}
  \right]
\end{equation}
is a real, normal matrix; Lemma~\ref{le:Tsym1} below shows that $T$ is symmetric if $2d | (p-1)$.
This occurs if and only if its eigenvalues, the diagonal entries of \eqref{eq:Dgp}, are real.
The eigenvalues of $T$ are $(p-1)/d$ and the Gaussian periods in \eqref{eq:Dgp}.

\begin{example}
For $d=2$, we can compute all absolute moments.
Recall that for $d=2$, \cite[Sec.~4.4]{SESUP} provides
\begin{equation}
  T =
  \begin{cases}
    \begin{bmatrix}
      \frac{p-5}{4} & \frac{p-1}{4} & \sqrt{\frac{p-1}{2}} \\
      \frac{p-1}{4} & \frac{p-1}{4} & 0 \\
      \sqrt{\frac{p-1}{2}} & 0 & 0
    \end{bmatrix} &\text{if $p \equiv 1 \pmod 4$},\\[25pt]
    \begin{bmatrix}
      \frac{p-3}{4} & \frac{p+1}{4} & 0 \\
      \frac{p-3}{4} & \frac{p-3}{4} & \sqrt{\frac{p-1}{2}} \\
      \sqrt{\frac{p-1}{2}} & 0 & 0
    \end{bmatrix} &\text{if $p \equiv 3 \pmod 4$},
  \end{cases}
\end{equation}
with eigenvalues $(p-1)/4$ and
\begin{equation}
  \eta_0 =
  \begin{cases}
    \dfrac{-1 \pm \sqrt{p}}{2} &\text{if $p \equiv 1 \pmod 4$}, \\[8pt]
    \dfrac{-1 \pm i\sqrt{p}}{2} &\text{if $p \equiv 3 \pmod 4$},
  \end{cases}\qquad
  \eta_1 =
  \begin{cases}
    \dfrac{-1 \mp \sqrt{p}}{2} &\text{if $p \equiv 1 \pmod 4$}, \\[8pt]
    \dfrac{-1 \mp i\sqrt{p}}{2} &\text{if $p \equiv 3 \pmod 4$}.
  \end{cases}
\end{equation}
Thus,
\begin{equation*}
  V_1(p) = |\eta_0| + |\eta_1| =
  \begin{cases} 
    \sqrt{p} &\text{if $p \equiv 1 \pmod 4$}, \\[5pt]
    \sqrt{p+1} &\text{if $p \equiv 3 \pmod 4$}.
  \end{cases}
\end{equation*}
More generally, using the binomial theorem, 
\begin{equation*}
    V_n(p) = \sum_{s=0}^1 |\eta_s|^n = |\eta_0|^n + |\eta_1|^n =
  \begin{cases}
   \displaystyle 2^{1-n} \sum_{j=0}^{\floor*{\frac{n}{2}}} \binom{n}{2j} p^{\frac{n}{2}-j} &\text{if $p \equiv 1 \pmod 4$}, \\[12pt]
    2^{1-n} \left( p+1 \right)^{n/2} &\text{if $p \equiv 3 \pmod 4$},
  \end{cases}
\end{equation*}
so $V_n(p) \sim 2^{1-n} p^{n/2}$ as $p \to \infty$.
\end{example}

\subsection{Structure of \texorpdfstring{$T$}{T}}
The real, normal matrix $T$ defined by \eqref{eq:Tgp} plays a central role in what follows.
It possesses a number of symmetries, some of which are subtle.

\begin{lemma}\label{le:Tsym1}
Let $T$ denote the matrix \eqref{eq:Tgp}.
 \begin{enumerate}
    \item If $d > 1$ and $2d \mid (p - 1)$, then $T$ is symmetric.

    \item For all $m$ and $n$ with $0 \leq m, n \leq d-1$,
         \begin{equation*}
        c_{0,m,n} = c_{0,a, b},
         \end{equation*}
	 where $a$ and $b$ are the least positive residues modulo $d$ of $-m$ and $n-m$, respectively.
	 We write $c_{0,m,n} = c_{0,-m,n-m}$ when no confusion can occur.

       \item If $d > 1$ and $j = (p-1)/2$, then for all $m$ and $n$ with $0 \leq m,n \leq d-1$,
         \begin{equation*}
        c_{0,m,n} = c_{0,a, b},
         \end{equation*}
	 where $a$ and $b$ are the least positive residues modulo $d$ of $n+j$ and $m+j$, respectively.
	 Similarly, we write $c_{0,m,n} = c_{0,n+j,m+j}$ when no confusion can occur.
  \end{enumerate}
\end{lemma}

\begin{proof}
  All equalities are over $\F_p$.

\smallskip\noindent(a) Recall that $\alpha$ is the least positive residue of $(p-1)/2$ modulo $d$.
  In this case, $\alpha = 0$, so we must show that $c_{0,m,n} = c_{0,n,m}$ when $0 \leq m, n \leq d-1$.
  Observe that $c_{0,m,n}$ is the number of solutions to
 $a + g^mb = g^n$ with $a,b \in \Gamma$; that is,
 $(-ab)^{-1} + (b^{-1})g^n = g^m$.
 Thus, $c_{0,m,n} = c_{0,n,m}$ if $-1 = g^{(p-1)/2} \in \Gamma$.

\smallskip\noindent(b)  
Observe that $c_{0,m,n}$ is the number of solutions to
$a + g^mb = g^n$, where $a,b \in \Gamma$; this is the number of solutions to
$ag^{-m} + b = g^{n-m}$, so (b) follows.
 
\smallskip \noindent(c)  
  Note that $c_{0,m,n}$ is the number of solutions to
$a + g^mb = g^n$ with $a, b \in \Gamma$;
this is the number of solutions to $ab^{-1} - g^n b^{-1} = -g^m$,
 or, equivalently, to
 $ab^{-1} + g^{n+(p-1)/2} b^{-1} = g^{m + (p-1)/2}$, so (c) holds.
\end{proof}

\subsection{Some lemmas}

We use the following lemma frequently in what follows.
\begin{lemma}\label{le:colsum}
    Let $0 \leq n \leq d-1$, then $\sum_{i=0}^d c_{0,i,n} = \sum_{i=0}^d c_{0,n,i} = \frac{p-1}{d}$.
\end{lemma}

\begin{proof}
  By Lemma~\ref{Lemma:SESUP}, $c_{0,m,n}$ is the number of solutions to $x + y = z$ with $x \in X_0$ and $y \in X_m$, where $z \in X_n$ is fixed. 
  Since $\{X_0, X_1, \ldots, X_d\}$ is a partition of $\F_p$,
  \begin{equation*}
    \sum_{i=0}^d c_{0,i,n} = 
    \big|
      \{ (x,y) \in X_0 \times \F_p : x + y = g^n \} \big|
    = \abs{X_0} = \frac{p-1}{d}.  
  \end{equation*}
  The equality $\sum_{i=0}^d c_{0,i,n} = \sum_{i=0}^d c_{0,n,i}$ follows from Lemma~\ref{le:Tsym1}.
\end{proof}

\begin{lemma}\label{thr:v2}
  If $d > 1$, then $V_2(p) = \frac{(d-1)p + 1}{d}$.
\end{lemma}

\begin{proof}
  Since $U$ is unitary,
  $\sum_{j=0}^{d-1} \left| \eta_j \right|^2 + \frac{p-1}{d} = p$.  Thus,
  \begin{equation*}
    V_2(p) = \sum_{j=0}^{d-1} \left| \eta_j \right|^2 = p - \frac{p-1}{d} = \frac{(d-1)p + 1}{d}. \qedhere
  \end{equation*}
\end{proof}

\begin{corollary}
  If $d > 1$, then $\sqrt{ \frac{(d-1)p +1}{d}} \leq V_1(p) \leq \sqrt{(d-1)p + 1}$.
\end{corollary}

\begin{proof}
 Since
 \begin{equation*}
      V_1(p)^2 
      = \bigg( \sum_{i=0}^{d-1} \left| \eta_i \right| \bigg)^2 
               = \sum_{i=0}^{d-1} \left| \eta_i \right| ^2 + \sum_{0 \leq j < k \leq d-1} \left| \eta_j \eta_k \right| 
               = V_2(p) + \sum_{0 \leq j < k \leq d-1} \left| \eta_j \eta_k \right|,
  \end{equation*}
  we get $\sqrt{V_2(p)} \leq V_1(p)$.  Lemma~\ref{thr:v2} now gives the desired lower inequality.
  Use Cauchy--Schwarz and Lemma~\ref{thr:v2} to obtain the upper inequality:
  \begin{align*}
    \bigg| \sum_{i=0}^{d-1} | \eta_i | \cdot 1 \bigg|  
    &= \bigg( \sum_{j=0}^{d-1} |\eta_j|^2 \bigg)^{1/2} \bigg(  \sum_{k=0}^{d-1} |1|^2   \bigg)^{1/2} 
    \leq \bigg( d\sum_{j=0}^{d-1} |\eta_j|^2 \bigg)^{1/2} \\
    &= \bigg( d \left( \frac{(d-1)p + 1}{d}\right) \bigg)^{1/2} 
    = \sqrt{(d-1)p + 1}. \qedhere
  \end{align*}
\end{proof}

The entries of $T$ are related to the number of points on modified Fermat curves.

\begin{lemma}\label{le:T}
  Let $0 \leq i,j,k < d$ and let $M_{i,j,k}$ be the number of $\F_p$-rational points $(x,y,z)$
  on the projective modified Fermat curve 
 \begin{equation} \label{eq:fermatcong}
    g^i x^d + g^j y^d = g^k z^d,
  \end{equation}
 defined over $\mathbb{P}^2(\F_p)$. Then
 \begin{equation}
    M_{i,j,k} = d^2 c_{i,j,k} + d(\delta_{j,k} + \delta_{i,k} + \delta_*),
  \end{equation}
 where
 \begin{equation*}
    \delta_* =
    \begin{cases}
      1 & \text{if $\frac{p-1}{2} \equiv i -j \pmod{d}$}, \\
      0 & \text{otherwise}.
    \end{cases}
 \end{equation*}
\end{lemma}

\begin{proof}
  Since $c_{i,j,k}$ is the number of solutions to $g^i a + g^j b = g^k$ with $a, b \in \Gamma = \langle g^d \rangle$, 
  we have $M_{i,j,k} = d^2 c_{i,j,k} + S$, where $S$ is the number of $\F_p$-rational solutions to \eqref{eq:fermatcong} with $xyz = 0$.
  If $x = 0$, then $(y/z)^d = g^{k-j}$ has a solution if and only if $j = k$.
  If $y = 0$, then $(x/z)^d = g^{k-i}$ has a solution if and only if $k = i$.
  Finally, if $z = 0$, then $(x/y)^d = g^{(p-1)/2 + j-i}$ has a solution if and only if $(p-1)/2 \equiv i - j \pmod d$.
  In each case, there are exactly $d$ such $\F_p$-rational solutions.
\end{proof}

Lastly, we recall the well-known bound of Hasse and Weil \cite{Hasse,Weil}:
\begin{theorem}[Hasse--Weil]\label{thm:hw}
 Let $E$ be a smooth, irreducible, projective curve of genus $\nu$ defined over a finite field $\F_q$.
  If $M$ is the number of $\F_q$-rational points on $E$, then
 \begin{equation*}
    | M - (q+1) | \leq 2 \nu \sqrt{q}.
 \end{equation*}
\end{theorem}

The Hasse--Weil bound and Lemma~\ref{le:T} can be used to obtain estimates on $c_{i,j,k}$, the number of solutions to $g^i a + g^j b = g^k$, in which $a, b \in \Gamma$.
In particular, the curve $M_{i,j,k}$ in Lemma~\ref{le:T} is a smooth, irreducible, projective plane curve of degree $d$, so its genus is $(d-1)(d-2)/2$ by the Pl\"{u}cker formula  \cite[Ch.~2, Sec.~4]{griffiths_harris}.

\section{Fourth absolute moment for fixed \texorpdfstring{$k$}{k}} \label{sec:fapmfk} 

In this section, we assume $k = (p-1)/d$ is fixed.
This corresponds to fixing the size of the superclasses instead of the 
number of superclasses (fixed $d$).
For a fixed $k$, we can calculate $V_4(p)$ explicitly for all but finitely many primes $p$.

\subsection{Circular pairs}
Let $\Phi$ be the subgroup of $\F_p^\times$ of order $k$.
Then $(p, k)$ is \emph{circular} if $k \mid (p-1)$ and
\begin{equation}
  |(\Phi a + b) \cap (\Phi c + d)| \leq 2
\end{equation}
for all $a, c \in \F_p^\times$ and $b, e \in \F_p$ with $\Phi a \neq \Phi c$ or $b \neq e$.
For a fixed $k$, the set of non-circular pairs $(p,k)$ is finite \cite{clay, fong_ke}.
More background on this material is in \cite{ke_kiechle_jctsa}. 

Here are some examples for those not familiar with the circularity condition.

\begin{example}
The pair $(p, 2)$ is circular for all odd primes $p$ since $\Phi = \{1, -1\}$.
\end{example}

\begin{example}
The pair $(p,3)$ is circular for all odd primes $p$.
  Fix a primitive root $g$ modulo $p$.
  Then $d = (p-1)/3$ and $\Phi = \{ 1, g^d, g^{2d}\}$.
  Suppose toward a contradiction that there exists $a, c \in \F_p^\times$ and $b, e, \in \F_p$ such that
 $\Phi a + b = \Phi c + e$,
  and $b \neq e$ or $\Phi a \neq \Phi c$.
  There are six cases to consider in $\Phi a + b = \Phi c + e$.
  Suppose that $a + b = cg^d + e$, $ag^d + b = cg^{2d}+e$, and $ag^{2d} + b = c + e$.
  Subtracting the first two equations gives
 $a(g^d - 1) = c(g^{2d} - g^d)$, so
 $a = cg^d$.
  Since $a + b = cg^d = e$, we have $b = e$ and hence $\Phi a = \Phi c$, a contradiction.
  The other five cases are similar.
\end{example}

\begin{example}
We claim that $(p, k) = (5, 4)$ is not circular. Let $a = 1$, $b = 0$, $c = 1$, and $e = 3$.
  In this case, $\Phi = \F_5^\times = \{ 1, 2, 3, 4 \}$.
	Moreover, $\Phi a + b = \Phi$ and $\Phi c + e = \{ 4, 5, 6, 7 \} = \{ 0, 1, 2, 4 \}$, so 
	$|(\Phi a + b) \cap (\Phi c + e) | = 3$, but $b \neq e$.
  Thus, $(p, k) = (5, 4)$ is not circular.
\end{example}

We now relate circularity with the diagonal entries of the $T$ matrix in \eqref{eq:Tgp}.

\begin{lemma}\label{le:fixkfirstcol}
  If $(p, k)$ is circular and $0 \leq m \leq d-1$, then $c_{0,m,m} \in \{0, 1, 2\}$ and $c_{0,m,m} = | X_m \cap (X_m + 1) |$.
\end{lemma}

\begin{proof}
  Let $\Phi = X_0 = \Gamma = \langle g^d \rangle$.
  By definition, $c_{0,m,m}$ is the number of solutions to $a + g^m b = g^m$ with $a, b \in \Phi$.
  If we let
 \begin{equation*}
    t'_m = \big| \{ (x,y) : x^d + g^m y^d = g^m, xy \neq 0 \} \big|,
 \end{equation*}
 then $t'_m = c_{0,m,m} d^2$.
  A rearrangement argument reveals that
 \begin{equation}
    t'_m = \big| \{ (x,y) : g^mx^d - g^my^d = 1, xy \neq 0\} \big|.
  \end{equation}
 If $g^m x^d = g^m y^d + 1$, then $g^mx^d \in \Phi g^m \cap (\Phi g^m + 1)$.
  Define $t = |\Phi g^m \cap (\Phi g^m + 1)|$.
  Since $\Phi$ is the group of $d$th power residues, it follows that $t'_m = td^2$, so $c_{0,m,m} = t$.
  Since $(p, k)$ is circular, we conclude that $t \in \{0, 1, 2\}$.
\end{proof}

Before we prove the next lemma, which uses a result from the field of planar nearrings, we provide a brief background on some notation used in the reference \cite{ke_kiechle_jctsa}.
As in the introduction, and noting that $\Gamma = \Phi$, consider $E_h^{g^m} = \{\Phi g^m + b \mid b \in \Phi h \}$.
We now construct a graph on $E_h^{g^m}$, $G$, with vertex set $V(G) = \Phi h$ and with edge set given by:
\[
	E(G) = \{ h_1 h_2 \mid h_1, h_2 \in \Phi h, h_1 \neq h_2, (\Phi g^m + h_1) \cap (\Phi g^m + h_2) \neq \varnothing \} \,.
\]
If $(p,k)$ is a circular pair, then $(\Phi g^m + h_1) \cap (\Phi g^m + h_2) \leq 2$ for $h_1 \neq h_2$.
Thus for any $h_1, h_2 \in V(G)$ with $h_1 \neq h_2$ and $h_1 h_2 \in E(G)$, the cardinality of $(\Phi g^m + h_1) \cap (\Phi g^m + h_2)$ is either $1$, in which case we call $h_1 h_2$ an \emph{odd} edge or it is $2$, in which case we call $h_1 h_2$ an \emph{even} edge.
Our notation for the graph associated to $E_h^{g^m}$ clashes with that used in \cite{ke_kiechle_jctsa}, so we make some further notes.
Recall that we have defined $\Gamma = \langle g^d \rangle$, whereas \cite{ke_kiechle_jctsa} use $\Gamma(E_h^{g^m})$ for what we have called $G$ above.
Moreover, the following dictionary between the usage in \cite{ke_kiechle_jctsa} on the left-hand side our usage on the right-hand side may be of use: $\phi \leftrightarrow g^d, r \leftrightarrow g^m, c \leftrightarrow h$.

\begin{lemma}\label{le:fixedkdiag}
Suppose that $(p, k)$ is circular with $2 \mid k$ and $0 \leq m \leq d - 1$. 
Then $c_{0,m,m} = 1$ if and only if $2^{-1} \in X_m$.
\end{lemma}

\begin{proof}
  Since $2 \mid (p-1)$, there is an $h \in \F_p^{\times}$ with $2h = 1$.
  As in the proof of Lemma~\ref{le:fixkfirstcol}, we have
 \begin{equation}
    c_{0,m,m} = t = |\Phi g^m \cap (\Phi g^m + 2h)| = |(\Phi g^m - h) \cap (\Phi g^m + h)|.
  \end{equation}
  Considering the discussion preceding the lemma, we must show that $\{-h, h\}$ is an odd edge in $E(G)$ if and only if $h \in X_m$.
  Note that $2 \mid k$ ensures that $-1 \in \Phi$ so that $h$ and $-h$ are vertices in $V(G)$.
  Moreover by \cite[Lem.~4.3,Thm.~4.4]{ke_kiechle_jctsa}, $\{-h, h\}$ is an odd edge in $V(G)$ if and only if $h \in \Phi g^m$. Since $X_m = g^m X_0 = \Phi g^m$, the lemma follows.
\end{proof}

With Lemma~\ref{le:fixedkdiag} and the fact that $c_{0,m,m} \in \{0, 1, 2\}$, we can now compute $V_4(p)$.

\subsection{Proof of Theorem \ref{thm:fixk}}
  Let $k = (p-1)/d$ and assume $(p, k)$ is circular. 
  
  \smallskip\noindent(a) Suppose $2 \mid k$; we must show that $V_4(p) = 3p(k-1) - k^3$.
  Examine the $(1,1)$ entry of $T^*T = UD^*DU^*$ and deduce that
 \begin{equation}
    \frac{p-1}{d} + (c_{0,0,0}^2 + \cdots + c_{0,d-1, 0}^2) = \frac{1}{p} \bigg( V_4(p) + \bigg( \frac{p-1}{d} \bigg)^3 \bigg) \,,
  \end{equation}
  from which we conclude
 \begin{equation}\label{eq:v4fixedksub}
    V_4(p) = p\big(k + (c_{0,0,0}^2 + \cdots + c_{0,d-1,0}^2)\big) - k^3.
  \end{equation}
 Lemma~\ref{le:colsum} gives
 \begin{equation}
    1 + \sum_{i=0}^{d-1} c_{0,i,0} = k \,,
  \end{equation}
 and by Lemma~\ref{le:fixedkdiag}, since $2 \mid k$ and $(p,k)$ is circular, there is only one $i$ such that 
 $0 \leq i \leq d-1$, and $c_{0,i,i} = 1$, where we have also used Lemma~\ref{le:Tsym1} part (b).
  Let $\gamma$ be the number of $2$s in the sequence $c_{0,0,0}, \ldots, c_{0,d-1,0}$.
  Lemma~\ref{le:colsum} ensures that $\gamma = (k-2)/2$, so
 \begin{equation}\label{eq:v4circlem10}
    \sum_{i=0}^{d-1} c_{0,i,0}^2 = 1 + 4\left( \frac{k-2}{2} \right) = 2k - 3.
  \end{equation}
 The desired formula now follows from \eqref{eq:v4fixedksub} and \eqref{eq:v4circlem10}:
 \begin{equation}
    V_4(p) = p(k + (2k-3)) - k^3 = 3p(k-1) - k^3.
  \end{equation}

  \noindent(b)  
 Suppose $2 \nmid k$; we must show that $V_4(p) = p(2k-1) - k^3$.
  Since $\gcd(2, k)=1$ and $k \mid (p-1)$, there is a subgroup $\Psi$ of $\F_p^\times$ of order $2k$.
  Because $\Phi \subseteq \Psi$, 
  \begin{equation*}
      ((\Phi g^m - h) \cap (\Phi g^m + h)) \subseteq ((\Psi g^m - h) \cap (\Psi g^m + h)).
  \end{equation*}
  In this case, we have assumed $(p,2k)$ is circular, so we may apply \cite[Thm.~4.4]{ke_kiechle_jctsa} to the graph of $ \{ \Psi g^m + b \mid b \in \Psi h \}$.
  If $(\Psi g^m -h) \cap (\Psi g^m +h) = \{a, b\}$ with $a \neq b$, so that $\{h, -h\}$ is an even edge, then \cite[Thm.~4.4]{ke_kiechle_jctsa} says we may take $a = g^m - h$ and $b = -g^m + h$.
  Since $2 \nmid k$, $\Phi$ has no elements of order $2$, so $-1 \not\in \Phi$, which implies $b \not\in ((\Phi g^m - h) \cap (\Phi g^m + h))$.
  Thus the cardinality of $((\Phi g^m - h) \cap (\Phi g^m + h))$ is at most $1$ and $c_{0,m,0} \in \{0, 1\}$ for $0 \leq m \leq d-1$.
  We conclude from Lemma~\ref{le:colsum} that
 \begin{equation}
    \sum_{i = 0}^{d-1} c_{0,i,0} = k-1,
    \qquad \text{so} \qquad
    \sum_{i=0}^{d-1} c_{0,i,0}^2 = k-1.
  \end{equation}
  The desired formula now follows from \eqref{eq:v4fixedksub}:
 \begin{equation*}
    V_4(p) = p(k + (k-1)) - k^3 =  p(2k-1) - k^3 .\qedhere
  \end{equation*}

  \section{Absolute moments for fixed \texorpdfstring{$d$}{d}}\label{Section:FixedD}
In this section, we assume that $d$ is fixed and
prove Theorem \ref{thm:v4d3} and \ref{thm:d4p1p8}.

\subsection{Proof of Theorem~\ref{thm:v4d3}}
Let $t_i = c_{0,i,i}$ for $0 \leq i \leq 2$.
Since $6 \mid (p-1)$, Lemma~\ref{le:Tsym1} ensures that $T$ is symmetric.
Thus, using parts (a) and (b) of Lemma~\ref{le:Tsym1}:
\begin{equation}\label{eq:Td3p4}
  T=
  \left[
    \begin{array}{ccc|c}
				t_0                   &  t_2        &  t_1        &  \sqrt{\frac{p-1}{3}}\\
				t_2                   &  t_1        &  c_{0,2,1}  &  0\\
				t_1                   &  c_{0,2,1}  &  t_2        &  0\\
				\hline
				\sqrt{\frac{p-1}{3}}  &  0          &  0          &  0\\
    \end{array}
  \right].
\end{equation}
To determine $c_{0,2,1}$, we use Lemma~\ref{le:colsum}, which gives 
\begin{equation*}
  t_0 + t_1 + t_2 + 1 = \frac{p-1}{d}
  \qquad \text{and} \qquad
  t_2 + t_1 + c_{0,2,1} = \frac{p-1}{d},
\end{equation*}
so that
\begin{equation}\label{eq:Td3p4_2}
  T=
  \left[
    \begin{array}{ccc|c}
				t_0                   &  t_2  &  t_1  &  \sqrt{\frac{p-1}{3}}\\
				t_2                   &  t_1  &  t_0  +  1                       &  0\\
				t_1                   &  t_0  +  1    &  t_2                     &  0\\
				\hline
				\sqrt{\frac{p-1}{3}}  &  0    &  0    &  0\\
    \end{array}
\right].
\end{equation}
This is sufficient to tackle the fourth absolute moment $V_4(p)$. 

\begin{proof}[Proof of Theorem~\ref{thm:v4d3}.]
 Examine the $(1,1)$ entry in $T^*T = UD^*DU^*$ and obtain
 \begin{equation} \label{eq:d3p4_1}
    t_0^2 + t_1^2 + t_2^2 + \frac{p-1}{3} = \frac{1}{p} \bigg( V_4(p) + \left( \frac{p-1}{3} \right)^3 \bigg).
  \end{equation}
 Computing the trace reveals that
 \begin{equation}
    \tr  T^*T = 3(t_0^2 + t_1^2 + t_2^2) + 2 \left( \frac{p-1}{3} \right) + 4t_0 + 2.
  \end{equation}
 Since the eigenvalues of $T^*T$ are the diagonal entries of $D^*D$, we also have
 \begin{align} \label{eq:d3p4eig}
    \tr  T^*T &= |\eta_0|^2 + |\eta_1|^2 + |\eta_2|^2 + \left(\frac{p-1}{3}\right)^2 \\
                    &= V_2(p) + \left(\frac{p-1}{3}\right)^2.
  \end{align}
 Using Theorem~\ref{thr:v2} and comparing \eqref{eq:d3p4_1} and \eqref{eq:d3p4eig} gives
 \begin{equation} \label{eq:d3v4}
    V_4(p) = \frac{1}{27}(10p^2 -20p + 1) - \frac{4}{3}pt_0.
  \end{equation}
 Since $t_0 = c_{0,0,0} = (M_{0,0,0} - 9)/9$
 by Lemma~\ref{le:T}, Hasse--Weil (Theorem \eqref{thm:hw}) yields
 \begin{equation} \label{eq:d3p4hasse}
    \frac{p - 2\sqrt{p} - 8}{9} \leq t_0 \leq \frac{p+2\sqrt{p} - 8}{9}
  \end{equation}
 because the Fermat curve $x^3 + y^3 = z^3$ has genus $1$.
  Using \eqref{eq:d3p4hasse} in \eqref{eq:d3v4} gives the desired inequalities. 
\end{proof}

Although the Hasse--Weil bound is, for elliptic curves in general, tight, it does not appear to be tight for modified Fermat curves of higher genus.
For example, A. Garc\'ia and Voloch improve Hasse's bound on these curves in some instances \cite{garcia_voloch}.

\subsection{Proof of Theorem~\ref{thm:d4p1p8}}
When $d = 4$, the form of the matrix $T$ in \eqref{eq:Tgp} depends on the least positive residue of $(p-1)/2$ modulo $4$, so there are two cases:
$p \equiv 1 \pmod 8$ or $p \equiv 5 \pmod 8$.

  \smallskip\noindent\textbf{Case 1.}
  Suppose that $p \equiv 1 \pmod 8$.  Then $T$ is symmetric by Lemma~\ref{le:Tsym1}.
  Let $t_i = c_{0,i,i}$ for $0 \leq i \leq 2$, where $c_{i,j,k}$ is as in Lemma~\ref{Lemma:SESUP}. Then
 \begin{equation}\label{eq:d4p4T}
    T  =
    \left[
      \begin{array}{cccc|c}
				t_0 & t_3 & t_2 & t_1 & \sqrt{\frac{p-1}{4}}\\
				t_3 & t_1 & t_x & t_x & 0\\
				t_2 & t_x & t_2 & t_x & 0\\
				t_1 & t_x & t_x & t_3 & 0\\
				\hline
				\sqrt{\frac{p-1}{4}} & 0 & 0 & 0 & 0\\
      \end{array}
    \right]
  \end{equation}
  with $t_x=c_{0,2,1}$ and where we used Lemma~\ref{le:Tsym1} to fill out the remaining entries.
  Lemma~\ref{le:colsum} gives the following system of equations:
 \begin{align*}
    t_0 + t_1 + t_2 + t_3 + 1 &= \frac{p-1}{4}, \\
    t_1 + t_3 + 2t_x &= \frac{p-1}{4}, \quad \text{and}\\
    2t_2 + 2t_x &= \frac{p-1}{4},
  \end{align*}
 from which it follows that
 \begin{equation}
    t_x = \frac{p + 7 + 8t_0}{24}.
  \end{equation}
 The trace of $T^*T$ is
 \begin{equation}\label{eq:d4p4trace}
    \tr  T^*T = t_0^2 + 3(t_1^2 + t_2^2 + t_3^2) + 6 \left(\frac{p+7+8t_0}{24}\right)^2 + \left( \frac{p-1}{2} \right) \,.
  \end{equation}
  Using eigenvalues of $T^*T$, we have
 \begin{align}\label{eq:d4p4eig}
    \tr  T^*T = V_2(p) + \left( \frac{p-1}{4} \right)^2 &= V_2(p) + \left(\frac{p-1}{4}\right)^2 \\
                                                              &= \frac{3p+1}{4} + \left(\frac{p-1}{4}\right)^2,
  \end{align}
 where we have used Theorem \ref{thr:v2}. Using \eqref{eq:d4p4trace} and \eqref{eq:d4p4eig}, we arrive at
 \begin{equation}\label{eq:d4p4x}
    t_0^2 + t_1^2 + t_2^2 + t_3^2 = \frac{1}{288} \big( 128t_0^2 - 16(7+p)t_0 + 5p^2 -2p+29\big).
  \end{equation}
 Computing $(1,1)$ entry on both sides of $T^*T = UD^*DU^*$ gives
 \begin{equation}\label{eq:d4p411}
    t_0^2 + t_1^2 + t_2^2 + t_3^2 + \frac{p-1}{4} = \frac{1}{p} \bigg( V_4(p) + \left( \frac{p-1}{4}\right)^3 \bigg).
  \end{equation}
 Finally, \eqref{eq:d4p4x} and \eqref{eq:d4p411} yield the exact formula
 \begin{equation}\label{eq:v4pd4}
    V_4(p) = \frac{1}{576} \big( 256pt_0^2 - (32p^2 + 224p)t_0 + p^3 + 167p^2 - 113p + 9 \big).
  \end{equation}
  Since $M_{0,0,0} = M_{4,1}$, the equality in the theorem follows from Lemma~\ref{le:T}, which gives $M_{0,0,0} = 16t_0 + 3d$.
  The lower bound \eqref{eq:LowerSpecial} follows from minimizing the quadratic in \eqref{eq:v4pd4}, which occurs at $t_0 = (p+7)/16$.

  \smallskip\noindent\textbf{Case 2.}
Suppose that $p \equiv 5 \pmod 8$.
  As in the proof of when $p \equiv 1 \pmod 8$, we let $t_i = c_{0,0,i}$ for $0 \leq i \leq 2$, where $c_{i,j,k}$ is as in Lemma~\ref{Lemma:SESUP}.
  In this case, Lemma~\ref{le:Tsym1} ensures that
 \begin{equation}\label{eq:d4p4T5}
    T  =
    \left[
      \begin{array}{cccc|c}
        t_0 & t_1 & t_2 & t_3 & 0\\
        t_x & t_x & t_3 & t_1 & 0\\
        t_0 & t_x & t_0 & t_x & \sqrt{\frac{p-1}{4}}\\
        t_x & t_3 & t_1 & t_x & 0\\
        \hline
        \sqrt{\frac{p-1}{4}} & 0 & 0 & 0 & 0\\
      \end{array}
    \right],
  \end{equation}
 where $t_x = c_{0,1,0}$.
 Comparing the $(1,1)$ entries in
 $TT^* = UDD^*U^*$
 gives
  \begin{equation}\label{eq:d4p411entry}
    t_0^2 + t_1^2 + t_2^2 + t_3^2 = \frac{1}{p} \bigg( V_4(p) + \left( \frac{p-1}{4} \right)^3 \bigg).
  \end{equation}
 On the other hand, take the trace in $TT^* = UDD^*U^*$ and deduce
 \begin{equation}\label{eq:d4p4tracep5}
    \frac{p-1}{2} + 6 t_x^2+3 t_0^2+3 t_1^2+t_2^2+3 t_3^2 = V_2(p) + \left( \frac{p-1}{4} \right)^2.
  \end{equation}
 The structure of $T$ above and Lemma~\ref{le:colsum} gives 
 \begin{align*}
    2t_0 + 2t_x + 1 &= \frac{p-1}{4}, \\
    2t_x + t_1 + t_3 &= \frac{p-1}{4}, \quad\text{and}\\
    t_0 + t_1 + t_2 + t_3 &= \frac{p-1}{4},
  \end{align*}
 from which we derive
 \begin{equation}\label{eq:d4p4tx}
    t_x = \frac{8t_2 + p - 5}{24}.
  \end{equation}
 Using \eqref{eq:d4p4tracep5} and \eqref{eq:d4p4tx}, and Theorem~\ref{thr:v2}, we obtain
 \begin{equation}
    t_0^2 + t_1^2 + t_2^2 + t_3^2 = \frac{1}{288} \big(128t_2^2 - 16(p-5)t_2 + 5p^2 + 22p + 53 \big).
  \end{equation}
 Compare this with \eqref{eq:d4p411entry} and get the exact formula
 \begin{equation}\label{eq:v4d4p5eq}
    V_4(p) = \frac{1}{576} \big(p^3+71 p^2+256 p t_2^2-32 (p-5) p t_2+79 p+9\big). \qed
  \end{equation}
  As in Case 1, Lemma~\ref{le:T} is used to provide $M_{0,0,2} = 16d$.
  Since $M_{4,2} = M_{0,0,2}$, the equality in the theorem follows.
  The lower bound \eqref{eq:LowerSpecial} follows from minimizing the quadratic in \eqref{eq:v4d4p5eq}, which occurs at $t_2 = (p-5)/16$.

\subsection{Proof of Theorem \ref{thm:all_d_intro}}
Suppose that $d > 2$ and let $\alpha$ denote the least positive residue of $(p-1)/2$ modulo $d$.
Thus, $\alpha = 0$ or $\alpha = d/2$, corresponding to whether $2d \mid (p-1)$ or not.
Let $\delta_{\alpha}$ equal $1$ if $2d \mid (p-1)$ and $0$ otherwise.

The proof of Theorem \ref{thm:all_d_intro} needs two more lemmas.

\begin{lemma}\label{le:v4eqgen}
  Let $d \mid (p-1)$. Then
  \begin{enumerate}
  	\item $ V_4(p) = p ( \frac{p-1}{d} ) + p\sum_{l=0}^{d-1} c_{0,l,0}^2 - ( \frac{p-1}{d} )^3$ and
  	\item $ V_4(p) = p ( \frac{p-1}{d} )\delta_{\alpha} + p\sum_{l=0}^{d-1} c_{0,0,l}^2 - ( \frac{p-1}{d} )^3$.
  \end{enumerate}
\end{lemma}

\begin{proof}
  (a) Lemma~\ref{Lemma:SESUP} implies the equality of
 \begin{equation}
    [T^*T]_{1,1} = \sum_{l=0}^{d-1} c_{0,l,0}^2 + \left( \frac{p-1}{d} \right)
  \end{equation}
 and 
 \begin{equation}
    [UD^*DU^*]_{1,1} = \frac{1}{p} \bigg( V_4(p) + \left( \frac{p-1}{d} \right)^3 \bigg).
  \end{equation}
 Thus,
 \begin{equation*}
    V_4(p) = p \left( \frac{p-1}{d} \right) + p \sum_{l=0}^{d-1} c_{0,l,0}^2 - \left( \frac{p-1}{d} \right)^3.
  \end{equation*}
  
  \noindent(b) This follows in a similar manner from $[TT^*]_{1,1} = [UDD^*U^*]_{1,1}$. \qedhere
\end{proof}

\begin{lemma} \label{le:trace}
	For $d > 2$, we have
	$\displaystyle\sum_{m=0}^{d-1} \sum_{n=0}^{d-1} c_{0,m,n}^2 = \frac{p^2 + (d^2-3d-2)p + 3d + 1}{d^2}$.
\end{lemma}

\begin{proof}
	Compute the trace of $T^*T$ from its matrix entries:
	\begin{equation}
		\tr T^*T = \sum_{m=0}^{d-1} \sum_{n=0}^{d-1} c_{0,m,n}^2 + 2 \left( \frac{p-1}{d} \right) \,,
	\end{equation}
	and from its eigenvalues:
	\begin{equation*}\small
		\tr T^*T 
		= \sum_{\ell=0}^{d-1} |\eta_\ell|^2 + \bigg( \frac{p-1}{d} \bigg)^2 = V_2(p) + \bigg( \frac{p-1}{d} \bigg)^2 = \frac{(d-1)p + 1}{d} + \bigg( \frac{p-1}{d} \bigg)^2 .
	\end{equation*}
	Now compare the results.
\end{proof}

We are now ready to complete the proof of Theorem~\ref{thm:all_d_intro}:

\smallskip\noindent(a)
	From Lemma~\ref{le:v4eqgen},
	\begin{equation*}
		V_4(p) = p \left( \frac{p-1}{d} \right)\delta_\alpha + p \sum_{l=0}^{d-1} c_{0,0,l}^2 - \left( \frac{p-1}{d} \right)^3.
	\end{equation*}
	Cauchy--Schwarz and Lemma~\ref{le:colsum} ensure that
	\begin{equation*}
		\sum_{l=0}^{d-1} c_{0,0,l}^2 
		\geq \frac{1}{d} \bigg( \sum_{l=0}^{d-1} c_{0,0,l} \bigg)^2 
		= \frac{1}{d} \left( \frac{p-1}{d} - \delta_\alpha \right)^2 
		= \frac{(p-1-\delta_\alpha d)^2}{d^3}.
	\end{equation*}
	This yields the desired result:
	\begin{align*}
		V_4(p) 
		\geq p\left( \frac{p-1}{d} \right) + p\left(\frac{(p-1-\delta_\alpha d)^2}{d^3}\right) - \left( \frac{p-1}{d} \right)^3 
		=  \frac{((\delta_\alpha d-1)p+1)^2}{d^3} . 
	\end{align*}

\smallskip\noindent(b) Lemma~\ref{le:v4eqgen} gives
    \begin{equation}\label{eq:pv4p}
		V_4(p) = p \left( \frac{p-1}{d} \right) + p \sum_{l=0}^{d-1} c_{0,l,0}^2 - \left( \frac{p-1}{d} \right)^3. 
    \end{equation}
    Define 
    \begin{equation*}
    K = \frac{1}{d^2} (p^2 + (d^2-3d-2)p + 3d + 1)
    = k^2+ (d-3) k+1 \in \Z
    \end{equation*}
to ease notation.
    We start with Lemma~\ref{le:trace} and make repeated use of Lemma~\ref{le:Tsym1}.
    All sums run over $\{0, 1,\ldots, d-1\}$, unless otherwise indicated, with equality of indices taken modulo $d$.  Then
    \begin{align}
      \sum_{m} c_{0,0,m}^2 
	&= K - \sum_{m \neq 0} \sum_{n} c_{0,m,n}^2  = K - \sum_{m \neq 0} \sum_{n \neq \alpha} c_{0,m,n}^2 - \sum_{m \neq 0} c_{0,m,\alpha}^2 \\
	&= K - \sum_{m \neq 0} \sum_{n \neq \alpha} c_{0,m,n}^2 - \sum_{m} c_{0,m,\alpha}^2 + c_{0,0,\alpha}^2 = \frac{K}{2} - \frac{1}{2} \sum_{m \neq 0} \sum_{n \neq \alpha} c_{0,m,n}^2 + \frac{c_{0,0,\alpha}^2}{2} \\
	&= \frac{K}{2} - \frac{1}{2} \sum_{n \neq \alpha} \sum_{\substack{m \neq 0 \\ m \neq n - \alpha}} c_{0,m,n}^2 - \frac{1}{2} \sum_{n \neq \alpha} c_{0,n-\alpha, n}^2 + \frac{c_{0,0,\alpha}^2}{2} \\
	&= \frac{K}{2} - \frac{1}{2} \sum_{n \neq \alpha} \sum_{\substack{m \neq 0 \\ m \neq n - \alpha}} c_{0,m,n}^2 - \frac{1}{2} \sum_{n} c_{0,n-\alpha, n}^2 + c_{0,0,\alpha}^2 \\
	&= \frac{K}{3} - \frac{1}{3} \sum_{n \neq \alpha} \sum_{\substack{m \neq 0 \\ m \neq n - \alpha}} c_{0,m,n}^2 + \frac{2c_{0,0,\alpha}^2}{3} \,. \label{eq:p41}
    \end{align}

	Substitute  \eqref{eq:p41} into \eqref{eq:pv4p} and use Cauchy--Schwarz to obtain
	{\small
	\begin{align*}
		V_4(p) 
		&= \frac{( (d-3)p^3 + (3d^2 (\delta_\alpha - 1) + d^3 - 2d + 9)p^2 - (3d^2 (\delta_\alpha -1) - d + 9)p + 3)}{3d^3} \\
		& \qquad + \frac{2pc_{0,0,\alpha}^2}{3} - \frac{p}{3} \sum_{n \neq \alpha} \sum_{\substack{m \neq 0 \\ m \neq n - \alpha}} c_{0,m,n}^2 \\
		&\leq \frac{( (d-3)p^3 + (3d^2 (\delta_\alpha - 1) + d^3 - 2d + 9)p^2 - (3d^2 (\delta_\alpha -1) - d + 9)p + 3)}{3d^3} \\
		&\qquad + \frac{2pc_{0,0,\alpha}^2}{3} - \frac{p}{3(d-1)(d-2)} \Bigg( \sum_{n \neq \alpha} \sum_{\substack{ m \neq 0 \\ m \neq n - \alpha  }} c_{0,m,n} \Bigg)^2 \,. \label{eq:pv4divcs}
	\end{align*}
}
	Using Lemma~\ref{le:colsum} and the fact that
	$c_{0,d,n} = 0$ for $n \neq 0$, we have
	\begin{align*}
	    \sum_{n \neq \alpha} \sum_{\substack{m \neq 0 \\ m \neq n - \alpha}} c_{0,m,n} &= \sum_m \sum_n c_{0,m,n} - \sum_m c_{0,m,\alpha} - \sum_n c_{0,n-\alpha,n} + 2c_{0,0,\alpha} \\
	    &= \left( \frac{(d-1)(p-1)}{d} - (1-\delta_{\alpha, 0}) \right) - \left( \frac{p-1}{d} - \delta_{\alpha, 0} \right) \\
	    &\qquad\qquad - \left( \frac{p-1}{d} - \delta_{\alpha, 0} \right) + 2c_{0,0,\alpha} \\
	    &= \left( \frac{(d-3)(p-1)}{d} \right) + (3\delta_{\alpha, 0}-1) + 2c_{0,0,\alpha} \,.
	\end{align*}

	Finally, we have the bound
	{\small
	\begin{align*}
		V_4(p) 
		&\leq \frac{( (d-3)p^3 + (3d^2 (\delta_\alpha - 1) + d^3 - 2d + 9)p^2 - (3d^2 (\delta_\alpha -1) - d + 9)p + 3)}{3d^3} \\
		&\qquad + \frac{2pc_{0,0,\alpha}^2}{3} - \frac{p}{3(d-1)(d-2)} \Bigg( \left( \frac{(d-3)(p-1)}{d} \right) + (3\delta_{\alpha, 0}-1) + 2c_{0,0,\alpha}\Bigg)^2 \,.
	\end{align*}
	}

	Lemma~\ref{eq:fermatcong} and Theorem~\ref{thm:hw} give
	\begin{equation}
		\abs{d^2 c_{0,0,\alpha} + 3d \delta_\alpha - (p+1)} \leq (d-1)(d-2) \sqrt{p},
	\end{equation}
	so 
	\begin{equation}\small
		\frac{ (p+1) - 3d \delta_\alpha- (d-1)(d-2) \sqrt{p}}{d^2} \leq c_{0,0,\alpha} \leq \frac{ (p+1) - 3d \delta_\alpha + (d-1)(d-2) \sqrt{p}}{d^2}. \label{eq:v4pmidhw}
	\end{equation}
	Maximizing our bound subject to this restriction and considering the cases $\alpha = 0$ and $\alpha = d/2$ separately 
	yields the desired result. \qed

\bibliography{GPM}
\bibliographystyle{amsplain}

\end{document}